\documentclass[a4paper]{amsart}
\usepackage{textcomp, amsmath, amssymb, amsthm}
\usepackage{graphicx}
\usepackage{enumerate}
\usepackage{mathrsfs}
\usepackage{float}
\usepackage[makeroom]{cancel}

\newtheorem*{theorem*}{Theorem}
\newtheorem{theorem}{Theorem}[section]
\newtheorem{lemma}[theorem]{Lemma}
\newtheorem*{lemma*}{Lemma}

\newtheorem{proposition}[theorem]{Proposition}
\newtheorem{conjecture}[theorem]{Conjecture}

\newtheorem{cor}[theorem]{Corollary}

\theoremstyle{definition} \newtheorem{remark}[theorem]{Remark}
\newtheorem{defin}[theorem]{Definition}

\DeclareMathOperator{\proj}{proj}

\def\rr{{\mathbb R}}
\def\R{{\mathbb R}}

\def\su{\subset}

\def\al{\alpha}

\def\ga{\gamma}

\def\de{\delta}
\def\De{\Delta}
\def\Om{\Omega}
\def\om{\omega}
\def\la{\lambda}
\def\ep{\varepsilon}

\def\ti{\widetilde}
\def\diam{{\rm diam}\, }
\def\dim{{\rm dim}\, }

\def\hau{\mathcal{H}}
\def\net{\mathcal{N}}
\def\calc{\mathcal{C}}
\def\cala{\mathcal{A}}
\def\calh{\mathcal{H}}
\def\ih{\nu}
\def\ib{\beta}

\def\lkb{\lesssim}
\def\gkb{\gtrsim}

\def\phi{\varphi}

\providecommand{\semmi}[1]{}
\begin{document}

\title[Dimension of Furstenberg-type subsets of affine subspaces ]{Hausdorff dimension of Furstenberg-type sets associated to families of affine subspaces }
\author{K. H\'era}

\address
{Institute of Mathematics, E\"otv\"os Lor\'and University, 
P\'az\-m\'any P\'e\-ter s\'et\'any 1/c, H-1117 Budapest, Hungary}

\email{herakornelia@gmail.com}

\keywords{Hausdorff dimension, affine subspaces, Furstenberg sets}

\subjclass[2010]{28A78, 05B30}

\thanks{This research was supported  
by the 
Hungarian National Research, Development and Innovation Office - NKFIH, 124749, and 
by the \'UNKP-17-3 New National Excellence Program as well as the ELTE Institutional Excellence Program (1783-3/2018/FEKUTSRAT) 
of the Hungarian Ministry of Human Capacities.
} 

\begin{abstract}
We show that if $B \su \rr^n$ and $E \su A(n,k)$ is a nonempty collection 
of $k$-dimensional affine subspaces of $\R^n$ such that every
$P \in E$ intersects $B$ in a set of Hausdorff dimension at least $\al$ with $k-1 < \al \leq k$,
then 
$\dim B \ge  \al +\dim E/(k+1)$, where $\dim$ denotes 
the Hausdorff dimension. 
This estimate generalizes the well known Furstenberg-type estimate 
that every $\al$-Furstenberg set in the plane has Hausdorff dimension at least $\al + 1/2$. 

More generally, we prove that if $B$ and $E$ are as above with $0 < \al \leq k$, 
then 
$\dim B \ge  \al +(\dim E-(k-\lceil \al \rceil)(n-k))/(\lceil \al \rceil+1)$. 
We also show that this bound is sharp for some parameters. 

As a consequence, we prove that for any $1 \leq k<n$, the union of any nonempty 
$s$-Hausdorff dimensional family of $k$-dimensional affine subspaces 
of $\R^n$ has
Hausdorff dimension at least $k+\frac{s}{k+1}$. 
\end{abstract}

\maketitle

\section{Introduction and statements of the main results}
\label{introd}
The following question arose from the work of Furstenberg \cite{Fu}.
Fix $0 < \al \leq 1$, and suppose that $F \su \rr^2$ is a compact set such that for
every $e \in S^1$ there is a line $L_e$ with direction $e$ such that 
$\dim (L_e \cap F) \geq \al$, where $\dim$ denotes the Hausdorff dimension. What is the smallest possible value of $\dim F$? 
Such sets are called $\al$-Furstenberg-sets. 
Wolff \cite{Wo99} gave the following partial answers to the question: 
For any $0 < \al \leq 1$, if $F \su \rr^2$ is an $\al$-Furstenberg set, then $\dim F \geq 2 \al,$ and $\dim F \geq \al + \frac{1}{2}$.
Moreover, for any $0 < \al \leq 1$ there exists a Furstenberg set with $\dim F = \frac{3\al}{2} + \frac{1}{2}$. 
In the $\al=1/2$ case Bourgain \cite{Bo03} improved the lower bound $1$ to $\dim F \geq 1 + c$ for some absolute constant $c > 0$ using the work
of Katz and Tao \cite{KaTa}.
However, the smallest possible value of the Hausdorff dimension of Furstenberg-sets is still unknown. 

Molter and Rela \cite{MR} considered the problem in higher generality: 
Let $0 < \al \leq 1$, $0 < s \leq 1$. 
We say that $F\su\R^2$ is an $(\al,s)$-Furstenberg set, if there is $E\su S^1$ with $\dim E=s$ such that for every $e\in E$ there is a line 
$L_e$ with direction $e$ with
$\dim (L_e \cap F )\geq \al$. In \cite{MR} it was proved that if 
$F\su\R^2$ is an $(\al,s)$-Furstenberg set, then $\dim F\ge 2\al-1+s$ and 
$\dim F\ge \al+\frac{s}{2}$.

In \cite{LuSt} Lutz and Stull investigated the generalized Furstenberg-problem using methods from information theory. 
They proved that if 
$F\su\R^2$ is an $(\al,s)$-Furstenberg set, then
$\dim F\ge \al+\min\{s,\al\}$. Their new bound is better than the one obtained in \cite{MR} whenever $\al, s <1$ and $s < 2\al$.

In \cite{HKM} the authors investigated Furstenberg-type sets associated to families of affine subspaces.
For any integers $1 \leq  k < n$, let $A(n, k)$ denote the space
of all $k$-dimensional affine subspaces of $\rr^n$. Let $0 < \alpha \leq k$, and $0 \leq s \leq (k+1)(n-k)$. 
We say that $B \su \rr^n$ is an $(\al,k,s)$-Furstenberg set, if there is $\emptyset \neq E \su A(n,k)$ with $\dim E = s$ such that $B$ has an at least $\al$-dimensional intersection with each $k$-dimensional affine subspace of the family $E$, that is, $\dim (B \cap P) \geq \al$ for all $P \in E$. What is the smallest possible value of $\dim B$ (as a function of $\al,  s, n, k$)? In \cite{HKM} it was proved that if
$B \su \rr^n$ is an $(\al,k,s)$-Furstenberg set, then 
$\dim B \geq 2\al-k+ \min\{s,1\}.$ The method used in \cite{HKM} generalizes the method of Wolff \cite{Wo99} yielding the lower bound $2 \al$ for classical plane $\al$-Furstenberg-sets. 

In this paper we also investigate Furstenberg-type sets associated to families of affine subspaces. Our method generalizes the method of Wolff \cite{Wo99} yielding the lower bound $\al + \frac{1}{2}$ for classical plane $\al$-Furstenberg-sets. 

The paper is organized as follows: In Section \ref{mainn} we state our main result (Theorem \ref{thm2}), and prove that the obtained bound for the Hausdorff dimension of $(\al,k,s)$-Furstenberg sets is sharp for some parameters. 
In Section \ref{coraffin} we list some results obtained for unions of affine subspaces. 
Section \ref{bas} contains the introductory steps, and Section \ref{count} contains the main arguments of the proof of our main result. The lengthy proofs of two important lemmas (Lemma \ref{pconst}, Lemma \ref{metricnet}) are postponed to Sections \ref{psepar} and \ref{linalg}. Section \ref{pure} contains the proofs of some relatively easy purely geometrical lemmas. 

\subsection{Notation and definitions}
The open ball of center $x$ and radius $r$ will be denoted by $B(x,r)$ or $B_{\rho}(x,r)$ if we want to indicate the metric $\rho$. 
For a set $U \su \rr^n$, $U_{\de}=\cup_{x \in U} B(x,\de)$ denotes the open $\de$-neighborhood of $U$, and $\diam (U)$ denotes the diameter of $U$. 
Let $s \geq  0$, $\de \in (0,\infty]$ and $A \su \rr^n$. By the $s$-dimensional Hausdorff $\de$-premeasure of $A$ we mean  
$$\hau^s_{\de}(A)= \inf \{\sum_{i=1}^{\infty} (\diam(U_i))^s : A \su \bigcup_{i=1}^{\infty} U_i, \ \diam(U_i) \leq \de \ (i=1,2,\dots)\}.$$ 
The $s$-dimensional Hausdorff measure of $A$ is defined as $\hau^s(A)=\lim_{\de \to 0} \hau^s_{\de}(A)$, and  
the $s$-dimensional Hausdorff content of $A$ is    
$$\hau^s_{\infty}(A)=\inf \{ \sum_{i=1}^{\infty} (\diam(U_i))^s : A \su \bigcup_{i=1}^{\infty} U_i\}.$$  
The Hausdorff dimension of $A$ is defined as 
$$\dim A=\sup\{s: \hau^s(X)>0\}=\sup\{s: \hau^s_{\infty}(X)>0\}.$$ 
For the well known properties of Hausdorff measures and dimension, see e.g. \cite{Fa}. 
For a finite set $A$, let $|A|$ denote its cardinality. 
We will use the notation $a \lesssim_{\alpha} b$ if $a \leq Cb$ where $C$ is a constant depending on $\alpha$. 
If it is clear from the context what $C$ should depend on, we may write only $a \lesssim b$. 
For any $x \in \rr$, the least integer greater than or equal to $x$ will be denoted by $\lceil x \rceil$. 
For $i \geq 2$ integer and $z_1, \dots, z_i \in \rr^n$, let $\De(z_1, \dots, z_i)$ denote the convex hull of the points $z_1, \dots, z_i$.

Let $1 \leq  k < n$ be integers, and let $A(n, k)$ denote the space
of all $k$-dimensional affine subspaces of $\rr^n$. 
Now we introduce the concept of natural metrics on $A(n,k)$. 
Let $G(n,k)$ denote the space of all $k$-dimensional linear subspaces of $\rr^n$. 
For $P_i=V_i + a_i \in A(n,k)$, where $V_i \in G(n,k)$ and $a_i \in V_i^{\perp}$, $i=1,2$, we put 
$$m(P_1,P_2)=\|\pi_{V_1}-\pi_{V_2}\| + |a_1-a_2|,$$ 
where $\pi_{V_i}: \rr^n \to V_i$ denotes the orthogonal projection onto $V_i$ ($i=1,2)$, and $\| \cdot \|$ denotes the standard operator norm. 
Then $m$ is a metric on $A(n,k)$, see \cite[p. 53]{Ma95}.  

\begin{defin}
Let $\rho$ be a metric on $A(n,k)$. We say that $\rho$ is a \emph{natural metric}, if 
$\rho$ and $m$ are strongly equivalent, that is, if  
there exist positive constants $K_1$ and $K_2$ such that, for every $P_1,P_2 \in A(n,k)$,
$K_1 \cdot m(P_1,P_2)\leq \rho (P_1,P_2) \leq K_2 \cdot m(P_1,P_2).$
\end{defin}

\subsection{The main results and their sharpness} 
\label{mainn}
Let $1 \leq  k < n$ be integers, and fix a natural metric $\rho$ on $A(n, k)$. 

Now we state one of our main results. 
\begin{theorem}
\label{thm1}
Let $k-1 < \al \leq k$, and $0 \leq s \leq (k+1)(n-k)$ be any real numbers.
Suppose that $B \su \rr^n$ is an $(\al,k,s)$-Furstenberg-set, that is, there exists $\emptyset \neq E \su A(n, k)$ with $\dim E = s$ such that for every $k$-dimensional affine subspace
$P \in E$, $\dim (P \cap B) \geq \al$. Then 
\begin{equation}
\label{geq}
\dim B \geq \al+ \frac{s}{k+1}. 
\end{equation}
\end{theorem}

\begin{remark}
\label{saf}
The following easy example demonstrates that Theorem \ref{thm1} can not hold if $\al \leq k-1$:
 
Let $\al \leq k-1$, and let $B$ be an $\al$-dimensional subset of a fixed $(k-1)$-dimensional affine subspace $V$. Take an $s$-dimensional family $E$ of $k$-dimensional affine subspaces containing $V$ such that $s>0$. Then $\dim (P \cap B) = \al$ for all $P$, and $\dim B = \al < \al + \frac{s}{k+1}$. 
\end{remark}

In the case of arbitrary $\al$, we prove the following.  
\begin{theorem}
\label{thm2}
Let $0 < \al \leq k$, and $0 \leq s \leq (k+1)(n-k)$ be any real numbers.  
Suppose that $B \su \rr^n$ is an $(\al,k,s)$-Furstenberg-set, that is, there exists $\emptyset \neq E \su A(n, k)$ with $\dim E = s$ such that for every $k$-dimensional affine subspace
$P \in E$, $\dim (P \cap B) \geq \al$. Then 
\begin{equation}
\label{smmm}
\dim B \geq  \al + \frac{s-(k-\lceil \al \rceil)(n-k)}{\lceil \al \rceil+1}.
\end{equation}

\end{theorem}

\begin{remark}
Note that Theorem \ref{thm2} implies Theorem \ref{thm1}, by $\lceil \al \rceil=k$ if $k-1 < \al \leq k$. 
\end{remark}

We claim that both Theorem \ref{thm1} and Theorem \ref{thm2} are sharp for some parameters. 
Namely, for any $0 < \al \leq k$, there exist families of affine subspaces $E_1, E_2 \su A(n,k)$, and generalized Furstenberg-sets 
$B_1, B_2 \su \rr^n$ associated to them such that 
$\dim B_1, \dim B_2$ equals the lower bound obtained from Theorem \ref{thm1} and Theorem \ref{thm2}, respectively. This is the content of the following two propositions. 

\begin{proposition}
\label{sthm1}
Let $0 < \al \leq k$ be any real number, and $m \in [0,n-k]$ integer. There exists an $(\al,k,s)$-Furstenberg-set $B \su \rr^n$ with $s= m(k+1)$ 
such that $\dim B = \al + m =  \al +  \frac{s}{k+1}$. 
\end{proposition}

\begin{proof}
Let $B \su \rr^n$ be a Borel set contained in a $(k+m)$-dimensional affine subspace $H$ with $0 < \hau^{\beta}(B) < \infty$, where $\beta=\al + m$. 

Then the Marstrand-Mattila slicing theorem \cite{Ma95} implies that for $\gamma_{k+m, k}$-almost all $k$-dimensional linear subspace $W$ of $H$, 
$$\hau^m(\{a \in W^{\perp}: \dim ( B \cap (W + a))=\beta -m =\al \})>0,$$
where  $\gamma_{k+m, k}$ denotes the natural measure on the Grassmannian $G(k+m,k)$. 
This yields an $s=(k+1)m$-dimensional family $E$ of $k$-planes intersecting $B$ in a set of dimension $\al$. 
Thus $B$ is an $(\al,k,s)$-Furstenberg-set, $\dim B = \beta = \al + m = \al +  \frac{s}{k+1}$ and we are done. 
\end{proof}

\begin{proposition}
\label{sthm2}
Let $0 < \al \leq k$ be any real number. There exists an $(\al,k,s)$-Furstenberg-set $B \su \rr^n$ with  $s= (k-\lceil \al \rceil)(n-k)$ such that 
 $\dim B = \al  = \al + \frac{s-(k-\lceil \al \rceil)(n-k)}{\lceil \al \rceil+1}$. 
\end{proposition}

\begin{proof}
If $\al > k-1$, then $\lceil \al \rceil = k$,  $s =0$ so the proposition is trivial. 

If $\al \leq k-1$, let $B \su \rr^n$ be a set of dimension $\al$ contained in an $m=\lceil \al \rceil$-dimensional affine subspace $V$. Take $E$ to be the family of all $k$-dimensional subspaces of $\rr^n$ containing $V$. Easy computation gives that the dimension of $E$ equals $(k-m)(n-k)=(k-\lceil \al \rceil)(n-k)=s$. 
Then $\dim (P \cap B) = \al$ for all $P \in E$, so $B$ is an $(\al,k,s)$-Furstenberg-set with $\dim B = \al$ and we are done. 
\end{proof}

\subsection{Results for unions of affine subspaces}
\label{coraffin}

In this section we investigate unions $B=\bigcup_{P \in E} P \su \rr^n$ of affine subspaces,
where $\emptyset \neq E \subset A(n,k)$. We are interested in the smallest value of $\dim B$ as a function of $\dim E$. 

The first such result for unions of affine subspaces is due to Oberlin. 
\begin{theorem}[Oberlin, \cite{Ob2}]
\label{ober}
Let $\emptyset \neq E \su A(n, k)$ be compact, and put $B=\bigcup_{P \in E} P \su \rr^n$. Then 
$$\begin{cases}
\dim B \geq  2k -k(n- k) + \dim E, & \ \text{if} \ \dim E \leq (k+1)(n-k) -k,\\ 
B \ \text{has positive Lebesgue-measure}, & \ \text{if} \ \dim E > (k+1)(n-k) -k.
\end{cases}$$
\end{theorem} 

\begin{remark}
It is easy to see using Howroyd's theorem (\cite{Ho}) that if Theorem \ref{ober} holds for compact $E \su A(n,k)$, 
then it also holds for Borel (in fact, for analytic) $E$. 
\end{remark}

In \cite{HKM} the authors proved the following. 

\begin{theorem}[H\'era-Keleti-M\'ath\'e, \cite{HKM}] 
\label{hkmthm}
Let $\emptyset \neq E \su A(n, k)$, and put $B=\bigcup_{P \in E} P \su \rr^n$. Then 
$$
\dim B \geq k+ \min\{\dim E,1\}. 
$$
\end{theorem}

Note that in the $k=n-1$ case Theorem \ref{ober} already implies the $\dim B \geq n-1+ \min\{\dim E,1\}$ bound (for analytic $E$). 

\begin{remark}
It is not hard to see (\cite{HKM}) that for any $\emptyset\neq E \su A(n,k)$ we have $\dim \left(\bigcup_{P \in E} P \right) \leq k + \dim E$. 
This implies in particular that for any $\emptyset\neq E \su A(n,k)$ with $\dim E \in [0,1]$, $\dim \left(\bigcup_{P \in E} P \right) = k + \dim E$. 
\end{remark}

In this paper we also obtain new bounds for unions of affine subspaces as a corollary, namely, by applying  
Theorem \ref{thm1} for $B=\bigcup_{P \in E} P \su \rr^n$ and $\al = k$.  

\begin{cor}
\label{cor1}
Let $\emptyset \neq E \su A(n, k)$, and put $B=\bigcup_{P \in E} P \su \rr^n$. Then 
$$
\dim B \geq k+ \frac{\dim E}{k+1}. 
$$
\end{cor}

Easy computation gives that our estimate in Corollary \ref{cor1} is better than the bounds obtained in Theorem \ref{ober} and Theorem \ref{hkmthm} if 
$$k+1 < \dim E < (k+1)(n-k)-k-1.$$
The following easy construction will show that the bound in Theorem \ref{hkmthm} is sharp if $0 \leq \dim E \leq k+1$, and 
the bound in Theorem \ref{ober} is sharp if $(k+1)(n-k)-k-1 \leq \dim E \leq (k+1)(n-k)-k$, 
so in this paper we improve the previously known bounds in the whole region where it is possible. 
The next proposition will also show that our estimate in Corollary \ref{cor1} is not far from being sharp in general, and that it is sharp for some parameters.
Figure \ref{graph} below illustrates this situation. 

\begin{proposition}
\label{propalm}
For any $0 \leq s \leq (k+1)(n-k)$ there exists $B=\bigcup_{P \in E} P \su \rr^n$ with $\emptyset \neq E \su A(n,k)$ Borel and $\dim E = s$ such that
$$\dim B = \left.
  \begin{cases}
    s-k \lceil \frac{s}{k+1} \rceil + 2k & \text{if }   \ \lceil \frac{s}{k+1} \rceil \geq \frac{k+s}{k+1}, \\
    k+\lceil \frac{s}{k+1} \rceil & \text{if } \ \lceil \frac{s}{k+1} \rceil \leq \frac{k+s}{k+1}.
  \end{cases}
\right.$$ 
\end{proposition}

\newsavebox{\gs}
\savebox{\gs}{$h(s)=
  \left\{ \begin{array}{cc}
    s-2 \lceil \frac{s}{3} \rceil + 4 & \text{if }   \ \lceil \frac{s}{3} \rceil \geq \frac{2+s}{3}, \\
    2+\lceil \frac{s}{3} \rceil & \text{if } \ \lceil \frac{s}{3} \rceil \leq \frac{2+s}{3}
  \end{array} \right.$}

\begin{figure}
\begin{center}
\includegraphics[width=0.9\linewidth]{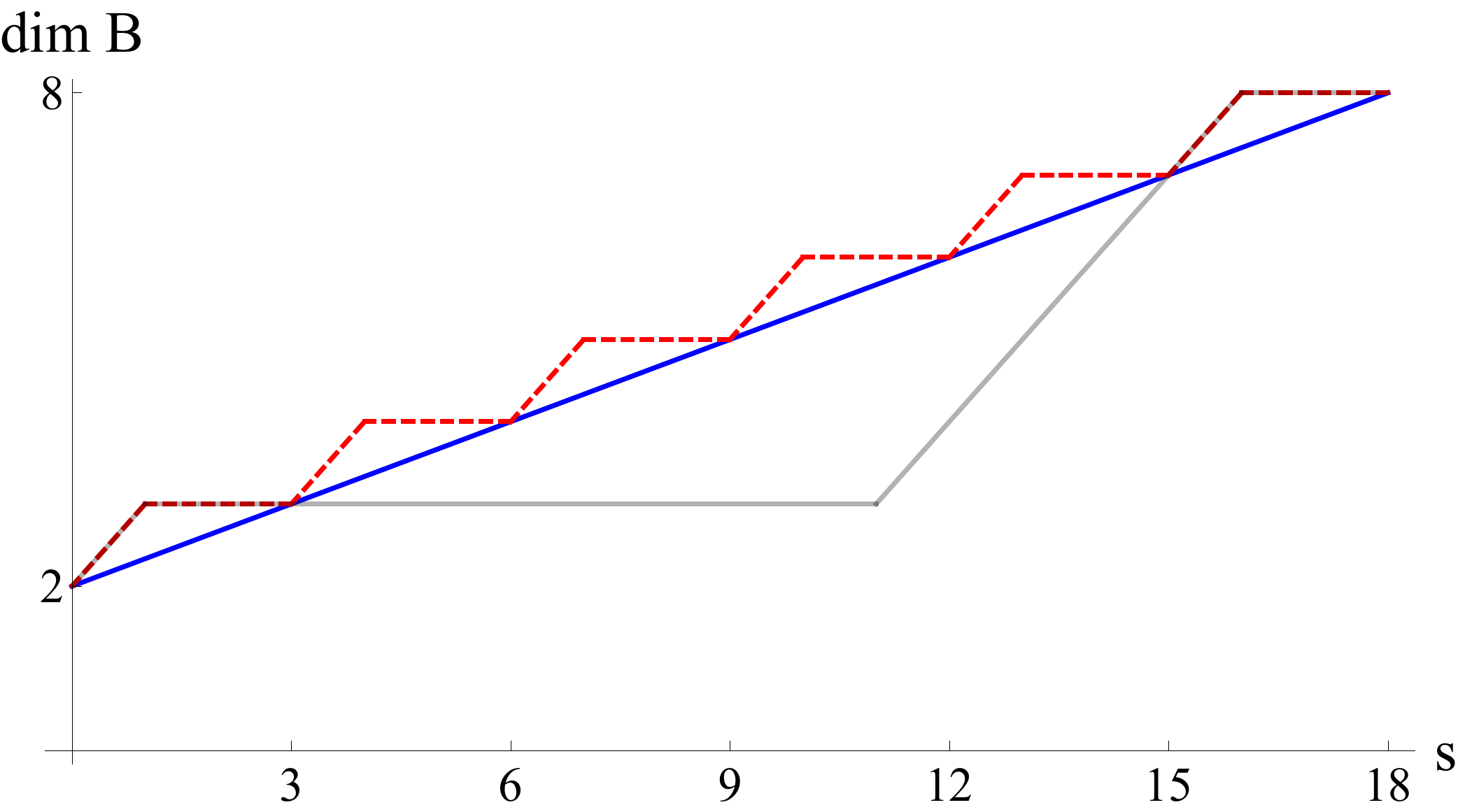}
\caption{The figure illustrates the graphs of three functions: 
the dimension bound obtained from the combination of Theorem \ref{ober} and  
Theorem \ref{hkmthm} ($f(s)=\max\{\min\{s-8,8\},2+\min\{s,1\}\}$, gray), 
the dimension bound obtained in Corollary \ref{cor1} ($g(s)=2+\frac{s}{3}$, blue), 
and the dimension of the construction in Proposition \ref{propalm}  (\usebox{\gs}, red dashed)
in the $k=2, n=8$ case.   
The regions where the red and gray graphs coincide show the intervals in which the bounds obtained in Theorem \ref{hkmthm} and Theorem \ref{ober} are sharp,
and the intersection points of the blue and red graphs indicate the values for which Corollary \ref{cor1} is sharp.}

\label{graph}
\end{center}
\end{figure}

\begin{proof}

Let $0 \leq s \leq (k+1)(n-k)$, and put $m=\lceil \frac{s}{k+1} \rceil \leq n-k$. 
First assume that $m \geq \frac{k+s}{k+1}$, and put $B=\rr^{m+k-1} \times A \times \{0 \}$, where $A \su \rr$ Borel with $\dim A = s-(k+1)(m-1)$. 
Then clearly, $B=\bigcup_{P \in E} P \su \rr^n$, where $E \su A(n,k)$ consists of all $k$-planes of $\rr^{m+k-1}$ translated by the elements of $A$. 
Clearly, $E$ is Borel and $\dim E = (k+1)(m-1)+ \dim A=s$. Moreover, $\dim B =m+k-1 + \dim A=s-km+2k$, so we are done. 

If $m \leq \frac{k+s}{k+1}$, put $B=\rr^{m+k-1} \times \rr \times \{0 \}$. Then it is easy to see that $B = \bigcup_{P \in E_0} P$, where $E_0 \su A(n,k)$ consists of all $k$-planes of $\rr^{m+k-1}$ translated by the elements of $\rr$, and clearly, $B = \bigcup_{P \in E} P$ for any $E_0 \su E \su A(m+k,k)$. 
Then $\dim B= m+k$, and for any $(k+1)m-k\leq s \leq (k+1)m$, there exists  $E \su A(n,k)$ Borel with with $\dim E = s$ such that $B=\bigcup_{P \in E} P,$ so we are done. 

\begin{remark}
With an easy modification of the proof, one can also construct a suitable $B=\bigcup_{P \in E} P \su \rr^n$ with $E \su A(n,k)$ compact. 
\end{remark}

\end{proof}

\begin{remark}
Clearly, the example in Proposition \ref{propalm} shows that for any $s \in [0,(k+1)(n-k)]$, there exists 
$B=\bigcup_{P \in E} P \su \rr^n$ with $E \su A(n,k)$ Borel and $\dim E = s$ such that
$$\dim B -\left(k+\frac{s}{k+1}\right) \leq  \frac{k}{k+1} < 1,$$
so our estimate in Corollary \ref{cor1} is not far from being sharp in general, the gap is less than $1$. 
Also, note that if $s=m(k+1)$ for some integer $m \in [0,n-k]$, then $\dim B = k + m = k + \frac{s}{k+1}$ in Proposition \ref{propalm}, 
which is exactly the bound obtained from Corollary \ref{cor1} for $B=\bigcup_{P \in E} P$ with $\dim E =s$, so Corollary \ref{cor1} is sharp for these parameters. 
\end{remark}

We also formulate the following conjecture. 

\begin{conjecture}
\label{conj}
Let $0 \leq s \leq (k+1)(n-k)$, $\emptyset \neq E \su A(n, k)$ with $\dim E = s$, and put $B=\bigcup_{P \in E} P \su \rr^n$. Then 
\begin{equation}
\dim B \geq \left.
  \begin{cases}
    s-k \lceil \frac{s}{k+1} \rceil + 2k & \text{if }   \ \lceil \frac{s}{k+1} \rceil \geq \frac{k+s}{k+1}, \\
    k+\lceil \frac{s}{k+1} \rceil & \text{if } \ \lceil \frac{s}{k+1} \rceil \leq \frac{k+s}{k+1}.
  \end{cases}
\right.
\label{eqconj}
\end{equation}
\end{conjecture}

Note that the dimension bound in \eqref{eqconj} equals the dimension of the construction in Proposition \ref{propalm}, so if Conjecture \ref{conj} is true, then it is sharp. The conjecture is motivated by the fact that Theorem \ref{hkmthm} and Theorem \ref{ober} imply that \eqref{eqconj} holds if $s \in [0,k+1] \cup [(k+1)(n-k)-k-1, (k+1)(n-k)]$. 

\section{The proof of Theorem \ref{thm2}, preparatory tools}
\label{bas}

Note that the statement of Theorem \ref{thm2} is trivially true if $s=0$. Let $s>0$. 

We introduce the following notations. Let $e_0=(0,\dots,0)$; let $e_1, \dots, e_n$ be the standard basis vectors of $\rr^n$, and 
let $V_0$ be the $k$-dimensional linear space spanned by $e_1, \dots, e_k$. Put $H_0=V_0^{\bot}$, and  
$H_i=e_i + H_0$. Then $H_i$ is an $n-k$-dimensional affine subspace for all $i=1,\dots,k$. 
We will use the following notation for the collection of $k$-planes which are positioned close to the horizontal subspace $V_0$: 
\begin{equation}
\calh=\{P \in A(n,k): P \cap H_i=\{x_i\} \ \text{for some} \ x_i \in [0,1]^n \ \forall \ i=0,1,\dots,k\}.
\label{hor}
\end{equation}

First we make several assumptions which do not restrict generality. The exact same arguments as the ones used in 
\cite[Lemma 3.3 and 3.4]{HKM} imply the following. 

\begin{lemma}
\label{assump}
\label{ass}
We can make the following assumptions in the proof of Theorem \ref{thm2}:
\begin{enumerate}[a)]
\item \label{bounded}   
$B  \su [0,1]^n$, $\hau^{\al}_{\infty} (B) \leq 1$;

\item\label{gdelta}
$B$ is a $G_\delta$ set, that is, a countable intersection of open sets;

\item
\label{comp}
\label{p0}
$E \su A(n,k)$ is compact, and $\hau^s(E)>0$;

\item
\label{a2}
\label{p1} $E \su \calh$;

\item
\label{a1}
there is $0 < \ep \leq 1$ such that for every $P\in E$,
$\hau^{\al}_{\infty}(P \cap B) \geq \ep.$

\end{enumerate}
\end{lemma}

Let us now fix $B$, $E$, $\ep$ with properties given by Lemma~\ref{ass}. 

We apply Frostman's lemma (see e.g. \cite{Ma95}) 
to obtain a probability measure $\mu$ on $A(n,k)$ (for which Borel and analytic sets are measurable) supported on $E$
for which
\begin{equation}
\label{Frost}
\mu (B(P,r)) \leq r^s 
\end{equation}
for all $r>0$ and all $P \in E$. 

\def\eps{\ep}

Now we turn to estimating the dimension of the set $B$. Our aim is to show that
$\hau^{t-\ga}(B)>0$ for any $\ga >0$, where $t=\al+\frac{s-(k-\lceil \al \rceil)(n-k)}{\lceil \al \rceil+1}$.

For our convenience, we will use net measures instead of Hausdorff measures. 
Let $\mathcal{D}$ denote the family of closed dyadic cubes in $\rr^n$, that is, cubes of the form
$$
\{x \in \rr^n: m_j 2^{-l} \leq x_j \leq (m_j + 1 ) 2^{-l}, m_j \in \mathbb{Z} \ (j=1,\dots,n), l \in \mathbb{Z} \}.
$$ 
Then for a set $X \su \rr^n$, $s \geq 0$ and $\de>0$, 
$$\net^s_{\de}(X)=\inf \{ \sum_{i=1}^{\infty} (\diam(D_i))^s : X \su \bigcup_{i=1}^{\infty} D_i, D_i \in \mathcal{D}, \diam(D_i) \leq \de \}$$ and 
$\net^s(X)=\lim_{\de \to 0} \net^s_{\de}(X)$. 
It is well known (see e.g. \cite{Fa}) that there exists a constant $b(n,s) \neq 0$ depending only on  $n$ and $s$ such that for every $X \su \rr^n$, 
$\hau^s(X) \leq \net^s(X) \leq b(n,s) \hau^s(X)$. 

We will show that  
\begin{equation}
\net^{t-\ga}(B)>0
\label{eqnet}
\end{equation}

for any $\ga >0$, where $t=\al+\frac{s-(k-\lceil \al \rceil)(n-k)}{\lceil \al \rceil+1}$, and this will imply  
$$\dim B\geq \al+\frac{s-(k-\lceil \al \rceil)(n-k)}{\lceil \al \rceil+1}.$$
Let $K$ be a positive integer such that
\begin{equation}
\sum_{l=K}^\infty 1/l^2 < \ep, 
\label{epsilon}
\end{equation} 
where $\ep$ is from \eqref{a1} of Lemma \ref{ass}, and such that 
\begin{equation}
k \cdot 2^{k+1} \cdot l^{2\phi} \leq 2^{l} \ \text{for all} \ l \geq K, 
\label{power}
\end{equation} 
where $\phi=\frac{\lceil \al \rceil}{\al-(\lceil \al \rceil-1)}$. 
For a dyadic cube $D_i \in \mathcal{D}$, let $r_i$ denote the side length of $D_i$. 

Let $B \su \bigcup_{i=1}^{\infty} D_i$ be any countable cover consisting of dyadic cubes 
such that $r_i \leq 2^{-K}$ for all $i$.
For any $l \geq K$, let 
$$J_l=\{ i : r_i =2^{-l} \}.$$ 
Let $R_l=\cup_{i \in J_l} D_i$, and $B_l=R_l \cap B$. Then $B=\cup_{l=K}^{\infty} B_l$. 

Our aim is to find a big enough subset of $B$ which is covered by cubes of the same side length and such that many of the affine subspaces of $E$ have big intersection with it. 
\def\rc{l}

\begin{lemma}
\label{prob2}
There exists an integer $\rc \geq K$ such that 
\begin{equation}
\label{B_l}
\mu\left(P \in E \colon \hau^{\al}_{\infty}(P \cap B_{\rc}) \geq 
\frac{1}{\rc^2 }\right) \geq \frac{1}{\rc^2 }.
\end{equation}
\end{lemma}

\begin{proof}
For $l \geq K$, let $$A_l=\left\{P \in E \colon \hau^{\al}_{\infty}(P \cap B_{l}) \geq \frac{1}{\rc^2}\right\}.$$ 
We need that the sets $A_l$ are $\mu$-measurable. 
It is easy to see that if $B\subset \R^n$ is compact, then the sets $A_l$ are closed. 
Therefore, if we had the extra assumption that $B \su \rr^n$ is compact, there would be no problem with the measurability.  

In general, the measurability of the above sets follows from the following lemma of M. Elekes and Z. Vidny\'anszky, which they 
have not published yet, a detailed proof can be found in \cite{HKM}. 
For the definition of analytic sets, see e.g. \cite{Fr}. 

\begin{lemma*}[Elekes-Vidny\'anszky]
\label{use}
If $X\subset \R^n$ is bounded $G_\delta$, then 
$$\{P \in A(n,k) \colon \hau^\alpha_{\infty}(P\cap X) > c \}$$ is analytic for all $\al>0$ and $c \geq 0$. 
\end{lemma*}

This lemma easily implies that $A_l$ is analytic, thus it is also $\mu$-measurable for each $l$. Assume that $\mu(A_{\rc}) < 1/{\rc}^2$ for all ${\rc} \geq K$.  
Since $\sum_{l=K}^\infty 1/l^2 < 1$, these sets $A_{\rc}$ cannot cover $E$. Therefore, there exists $P\in E$
such that $\hau^{\al}_{\infty}(P \cap B_{\rc})< 1/l^2$ for every $l \geq K$, and thus
$\hau^{\al}_{\infty}(P \cap B)< \sum_{l=K}^\infty 1/l^2<\eps$ by \eqref{epsilon}, which contradicts \eqref{a1} of Lemma \ref{ass}. 
\end{proof}

Fix the integer $\rc$ obtained from Lemma \ref{prob2} and put 

\begin{equation}
\label{tie}
\ti{E}=A_{\rc}=\left\{P \in E \colon \hau^{\al}_{\infty}(P \cap B_{\rc}) \geq \frac{1}{\rc^2} \right\}, 
\end{equation}

\begin{equation}
\label{eqj}
J=J_l.
\end{equation}

We have 
\begin{equation}
\label{est}
\mu( \ti{E}) \geq \frac{1}{\rc^2} \ \textrm{and} \ \hau^{\al}_{\infty}(P \cap B_{\rc})\geq \frac{1}{\rc^2}
\end{equation}
for every $P \in \ti{E}$ by Lemma \ref{prob2}. 

\section{The proof of Theorem \ref{thm2}, main argument}
\label{count}

The main idea of the proof of Theorem \ref{thm2} is to discretize the problem in an appropriate way, and to count the cardinality of a suitably defined finite set in two different ways. 
This idea was originally used in \cite{Wo99} to prove the lower estimate $\al + \frac{1}{2}$ for the Hausdorff dimension of classical  $\al$-Furstenberg-sets in the plane. 
We needed to generalize the idea in \cite{Wo99} to fit into our context. See also \cite{MR}.

Recall that 
\begin{equation}
\ti{E}=\left\{P \in E \colon \hau^{\al}_{\infty}(P \cap B_{\rc}) \geq \frac{1}{\rc^2}\right\}, 
\mu( \ti{E}) \geq \frac{1}{\rc^2}
\label{eqremind}
\end{equation} 
 by \eqref{tie} and \eqref{est}. 

We introduce the following notations. Put 
\begin{equation}
\de=2^{-l}, \la=\frac{1}{l^2}.
\label{eqdela}
\end{equation}
 
Note that using \eqref{power}, $\la < 1$, and $\phi =\frac{\lceil \al \rceil}{\al-(\lceil \al \rceil-1)}\geq 1$, we have 
\begin{equation}
\de \leq \frac{1}{2^{k+1}k} \cdot \la^{\phi} \leq \la.
\label{eqscales}
\end{equation}
 We will work with these two scales $\de$ and $\la$, 
and their relation \eqref{eqscales} will be important in the proofs. 

First we choose a maximal $\de$-separated set of affine subspaces from $\ti{E}$. 
This means, let 
$P_1,\dots, P_M \in \ti{E}$ with $\rho(P_i,P_j) \geq \de$ for every $i \neq j$, 
where $\rho$ indicates the given metric on $A(n,k)$, and such that $M$ is maximal. 
Then by the maximality of $M$, we have 
$$ \ti{E} \su \bigcup_{i=1}^{M} B_{\rho}(P_i,\de) \su A(n,k).$$ 
Thus by \eqref{eqdela}, \eqref{est} and \eqref{Frost}, 
$$ \la= \frac{1}{\rc^2} \leq \mu( \ti{E}) \leq \sum_1^M \mu( B_{\rho} (P_i, \de)) \leq M \de^s.$$

This implies
\begin{equation}
\label{direct}
M \geq \la\de^{-s}.
\end{equation}
Fix such a maximal $\de$-separated set 
\begin{equation}
E'=\{P_1,\dots, P_M \}.
\label{deltasep}
\end{equation}

Put 
\begin{equation}
m=\lceil \al \rceil \leq k, \ \ep = \al-(m-1) >0. 
\label{eqmep}
\end{equation}
We will use the following geometric lemma, which is a major element of the proof. 
The orthogonal projection onto the horizontal $k$-plane $V_0$ will be denoted by $\proj_0$. 

\begin{lemma}
\label{pconst}
For each $P \in \ti{E}$, there exist compact sets $T_0(P),\dots,T_m(P) \su [0,1]^n \cap P$ with the following properties: 
\begin{enumerate}[(1)]
	\item
	\label{w1}
	For any $z_0 \in \proj_0(T_0(P)),\dots,z_m \in \proj_0(T_m(P))$, $\hau^m(\De(z_0,\dots,z_m)) \gkb_{n,k,\al} \la^{\phi},$ 

	\item 
	\label{w2}
	$\hau^{\al}_{\infty}(T_i(P) \cap B_l) \gkb_{n,k,\al}  \la^{\psi}$ for all $i=0,\dots, m$, 
\end{enumerate}
where $\phi= \frac{m}{\ep}$, $\psi=1+\frac{mk}{\ep},$ and $\ep$ is from \eqref{eqmep}. 

\end{lemma}
We postpone the proof of Lemma \ref{pconst} to Section \ref{psepar}. 

Now we turn to defining a finite set $A$
using the subsets $T_0(P),\dots,T_m(P) \su [0,1]^n \cap P$.

Recall from \eqref{eqj} that $J=J_l=\{i: r_i=2^{-l}\}$, where $r_i$ denotes the side length of the dyadic cube $D_i$, and $B_l \su \bigcup_{i \in J} D_i$. 
Let
\begin{align}
\label{kispi}
A=\big\{& (j_0,j_1,\dots,j_m,i) \in J \times \dots \times J \times \{1,\dots,M\}:  \\
& \nonumber T_r(P_i) \cap B_l \cap D_{j_r} \neq \emptyset \ \forall r=0,\dots,m \big\}. 
\end{align}

We will bound the cardinality of $|A|$ from below and from above, this will yield a dimension estimate for $B$. 

\subsection{Lower bound on $|A|$}
\hfill\\
We will prove that 
\begin{equation}
|A| \gkb_{n,k,\al} \de^{-(s+\al \cdot (m+1))}\cdot \la^{1+\psi \cdot (m+1)},
\label{eqlower}
\end{equation}
where $\psi=1+\frac{mk}{\ep}$. 

To verify \eqref{eqlower} first we fix $i \in \{1,\dots,M\}$, 
we will count how many $(j_0,j_1,\dots,j_k)$'s there are at least such that $(j_0,j_1,\dots,j_k,i) \in A$. 

Fix $r \in \{0,\dots,m\}.$ 
We claim that there are at least $\approx \de^{-\al}\la^{\psi}$ many dyadic $\de$-cubes $D_{j_r}$ such that 
$T_r(P_i) \cap B_l \cap D_{j_r} \neq \emptyset$. To verify this, put
$$A_r=\{j_r \in J: T_r(P_i) \cap B_l \cap D_{j_r} \neq \emptyset \}.$$ 
Then by $$T_r(P_i) \cap B_l \su \bigcup_{j_r \in A_r}  D_{j_r}$$ and \eqref{w2} of Lemma \ref{pconst} we have
$$\la^{\psi} \lkb_{n,k,\al} \hau^{\al}_{\infty}( T_r(P_i) \cap B_l ) \leq |A_r| \cdot (\sqrt{n}\cdot\de)^{\al} \lkb |A_r| \cdot \de^{\al},$$
thus $|A_r| \gkb_{n,k,\al} \de^{-\al}\la^{\psi}$
for all $r=0,\dots,m$.

This implies 
$$|A| \geq M \prod_{r=0}^m |A_r| \gkb_{n,k,\al} M \cdot  \de^{-\al(m+1)}\la^{\psi(m+1)}.$$ 

Then by \eqref{direct} we have 

\begin{equation}
|A| \gkb_{n,k,\al} \de^{-(s+\al(m+1))}\la^{1+\psi(m+1)},
\end{equation}
where $\psi=1+\frac{mk}{\ep}$, and this proves \eqref{eqlower}. 

\subsection{Upper bound on $|A|$}
\hfill\\
We will prove that 
\begin{equation}
| A | \lkb_{n,k,\al} |J|^{m+1} \de^{-(k-m)(n-k)} \la^{-\phi (k+1)(n-k)},
\label{equpper}
\end{equation}
where $\phi= \frac{m}{\ep}$.

Fix $(j_0,\dots,j_m) \in J \times \dots \times J$, we will bound the amount of $i \in \{1,\dots,M \}$ with $(j_0,j_1,\dots,j_m,i) \in A$ from above. 
Let $c_0,\dots,c_m$  denote the centers of the fixed dyadic cubes $D_{j_0},\dots, D_{j_m}$, respectively, and $v_r=\proj_0(c_r)$, $r=0,\dots,m$.
Put 
\begin{equation}
A_{j_0,\dots,j_m}=\big\{i \in \{1,\dots,M \}: T_r(P_i) \cap B_l \cap D_{j_r} \neq \emptyset, r=0,\dots,m \big\}.
\label{eqae}
\end{equation}

First we prove the following easy lemma. 

\begin{lemma}
\label{lem}
If $A_{j_0,\dots,j_m} \neq \emptyset$, then 
$\hau^m(\De(v_0,\dots,v_m)) \gkb \la^{\phi}$,
where $\phi= \frac{m}{\ep}$. 
\end{lemma}

\begin{proof}
Since $T_r(P_i) \cap B_l \cap D_{j_r} \neq \emptyset$ for all $r=0,\dots,m$, we obtain using \eqref{w1} of Lemma \ref{pconst} that there are 
$y_0 \in \proj_0(D_{j_0}),\dots,y_m \in \proj_0(D_{j_m})$ such that $\hau^m(\De(y_0,\dots,y_m)) \gkb \la^{\phi}$, where $\phi= \frac{m}{\ep}$. 
 
We claim that if we perturb the points $y_i$ within the $k$-dimensional $\de$-cube $\proj_0(D_i)$, the measure of the $m$-simplex generated by them remains large. This is the content of the next lemma.  

\begin{lemma}[Perturbation]
\label{simplexkp}
Let $f: (0,1) \to \rr$ be a function such that $x \leq \frac{m}{2^{m+1}k}\cdot f(x)$ for all $0 < x < 1$. 
Let $D_0,\dots,D_m \su [0,1]^k$ be axis-parallel cubes of side length $x$ for some $0 < x < 1$.  
Assume that there are $y_0 \in D_0,\dots,y_m \in D_m$ such that $\hau^m(\De(y_0,\dots,y_m)) \geq f(x)$.
Then for any $z_0 \in D_0,\dots,z_m \in D_m$, $\hau^m(\De(z_0,\dots,z_m)) \geq \frac{f(x)}{2^{m+1}}$. 
\end{lemma}

We postpone the proof of Lemma \ref{simplexkp} to Section \ref{pure}. 

We apply Lemma \ref{simplexkp} for $x=\de$, the $\de$-cubes $\proj_0(D_{j_0}),\dots,\proj_0(D_{j_m})$, for $v_i$ in place of $z_i$, 
and with $f(x)=\frac{1}{\log^{2\phi}_2(1/x)}=\la^{\phi}$. By \eqref{eqscales}, 
$x\leq \frac{1}{2^{k+1}k} f(x) \leq \frac{m}{2^{m+1}k}f(x)$, so the lemma indeed can be applied. 
We obtain that $\hau^m(\De(v_0,\dots,v_m)) \gkb_{\al} \la^{\phi}$ and we are done with the proof of Lemma \ref{lem}. 
\end{proof}

Now we give an upper bound on $|A_{j_0,\dots,j_m}|$. Assume that $A_{j_0,\dots,j_m} \neq \emptyset$. 

Put 
\begin{equation}
\De=\De(v_0,\dots,v_m).
\label{eqadelta}
\end{equation} 
By Lemma \ref{lem}, $\hau^m(\De) \gkb \la^{\phi}$, where $\phi=\frac{m}{\ep}$. 
 
Let $\mathcal{A} \su A(n,k)$ denote the set of all $k$-planes of $\calh$ intersecting $D_r$ for all $r=j_0,\dots,j_m$. That is, 
\begin{equation}
\mathcal{A}=\{P \in \calh: D_{j_r} \cap P \neq \emptyset \ (r=0,\dots,m) \}.
\label{eqcala}
\end{equation}

\begin{remark}
\label{ntm}
Note that by definition, $T_r(P) \su P$ for each $r=1,\dots,m$, thus $\{P_i: i \in A_{j_0,\dots,j_m}\} \su \mathcal{A}$ is a finite $\de$-separated subset. 
\end{remark}

\begin{remark}
\label{remrem}
We remark that the condition $D_{j_r} \cap T_r(P) \cap B_l \neq \emptyset$ was intentionally weakened to $D_{j_r} \cap P \neq \emptyset$. We will use this weaker condition to give an upper bound on $|A_{j_0,\dots,j_m}|$. 
\end{remark}

We will need the following geometrical lemma. 

\begin{lemma}
\label{metricnet}

The metric space $(\mathcal{A},\rho)$ possesses a finite $\ep$-net $\calc \su \cala$ of cardinality $\approx \de^{-(k-m)(n-k)}$ 
with $\ep=\frac{K\de}{\hau^m(\De)}$, where $K$ is a constant depending only on $n,k,\al$.

That is, there exists a collection of $k$-planes 
$\calc \su \cala$ with $|\calc| \approx \de^{-(k-m)(n-k)}$ such that 
for any  $P \in \calh$ with $D_{j_r} \cap P \neq \emptyset$ for all $r=0,\dots,m$,
 $$P \in \bigcup_{C \in \calc} B_{\rho}\left(C, \frac{K\de}{\hau^m(\De)}\right),$$
where  $K$ is a constant depending only on $n,k,\al$.
\end{lemma}

We postpone the proof of Lemma \ref{metricnet} to Section \ref{linalg}. 

By Remark \ref{ntm} it is clear that 
$$B_{\rho}\left(P_i,\frac{\de}{3} \right) \cap B_{\rho}\left(P_j,\frac{\de}{3} \right) = \emptyset$$ 
for any $i,j \in A_{j_0,\dots,j_m}$ with $i \neq j$. 

Moreover, Lemma \ref{metricnet} implies that 
\begin{equation}
\bigcup_{i \in A_{j_0,\dots,j_m}} B_{\rho}\left(P_i,\frac{\de}{3} \right) \su 
\bigcup_{C \in \calc} B_{\rho}\left(C,   \frac{K \cdot \de}{\hau^m(\De)} + \frac{\de}{3} \right) \su A(n,k).
\label{ballunion}
\end{equation}

It is easy to see that there is a natural Radon measure $\ga$ on $A(n,k)$ such that the measure of any small enough $\de$-ball (in a natural metric) is comparable with 
$\de^{(k+1)(n-k)}$, that is, there exist $C_1,C_2$ such that 
$$C_1 \de^{(k+1)(n-k)} \leq \ga(B_{\rho}(P,\de)) \leq C_2 \de^{(k+1)(n-k)}.$$

This and \eqref{ballunion} imply 
$$|A_{j_0,\dots,j_m}| \cdot \de^{(k+1)(n-k)} \lkb |\calc|\frac{\de^{(k+1)(n-k)}}{\hau^m(\De)^{(k+1)(n-k)}},$$

thus
\begin{equation}
|A_{j_0,\dots,j_m}| \lkb |\calc| \cdot \frac{1}{\hau^m(\De)^{(k+1)(n-k)}} \lkb \frac{1}{\de^{(k-m)(n-k)}} \cdot \frac{1}{\hau^m(\De)^{(k+1)(n-k)}}.
\label{eqabe}
\end{equation}

We obtained that for any fixed dyadic cubes $D_{j_0},\dots, D_{j_m}$, $|A_{j_0,\dots,j_m}|$ can be bounded above as in \eqref{eqabe}, 
and clearly, 

$$|A| \leq \prod_{r=0}^{m}|\{j_r \in J\}| \cdot |A_{j_0,\dots,j_m}|,$$
thus by \eqref{eqabe} and Lemma \ref{lem}, 

$$|A|  \lkb |J|^{m+1}\frac{1}{\de^{(k-m)(n-k)}} \cdot \frac{1}{\hau^m(\De)^{(k+1)(n-k)}} \leq $$
$$|J|^{m+1} \frac{1}{\de^{(k-m)(n-k)}} \cdot \frac{1}{\la^{\phi (k+1)(n-k)}}
$$ 
with $\phi=\frac{m}{\ep}$, which proves \eqref{equpper}. 

Now we combine \eqref{eqlower} and \eqref{equpper} to obtain 
$$\de^{-(s+\al \cdot (m+1))}\la^{1+ \psi \cdot (m+1)} 
\lkb_{n,k,\al} |A| \lkb_{n,k,\al} |J|^{m+1} \cdot \de^{-(k-m)(n-k)} \cdot \la^{-\phi (k+1)(n-k)},$$ 
where $\psi=1+\frac{mk}{\ep}, \phi=\frac{m}{\ep}$. 
Thus 
\begin{equation}
\label{eqcomb}
|J| \gkb_{n,k,\al}  \de^{-(\frac{s-(k-m)(n-k)}{m+1}+\al)}\la^{\frac{1}{m+1}+ \psi +\phi \frac{(k+1)(n-k)}{m+1}} =
\de^{-(\frac{s-(k-m)(n-k)}{m+1}+\al)}\la^{\zeta} 
\end{equation}
with $\zeta=\frac{1}{m+1}+ \psi +\phi \frac{(k+1)(n-k)}{m+1}=\frac{1}{m+1}+ 1+\frac{mk}{\ep} +\frac{m}{\ep} \frac{(k+1)(n-k)}{m+1}$. 

Now we turn to estimating the $u$-dimensional net measure of $B$, where $u=\al+\frac{s-(k-m)(n-k)}{m+1}-\ga$ for some fixed $\ga>0$. Then \eqref{eqcomb} implies that 
$$\sum_{i=1}^{\infty} (\diam(D_i))^u \geq \sum_{i \in J} (\diam(D_i))^u \gkb \sum_{i \in J} \de^u 
= |J| \de^u \gkb \de^{-\ga}\la^{\zeta}.
$$

By \eqref{eqdela}, this means, 
$$\sum_{i=1}^{\infty} (\diam(D_i))^u \gkb 2^{l\cdot \ga}l^{-2\zeta}.$$
Then we have 
$$\inf_{\substack{B \su \bigcup_{i=1}^{\infty} D_i, r_i \leq 2^{-K} }} 
\sum_{i=1}^{\infty} (\diam(D_i))^u \gkb
\inf_{\rc \ge K} 2^{l \cdot \ga}l^{-2\zeta} \gkb_{\ga} 1$$
proving that $\net^u(B)>0$, and we are done with the proof of Theorem \ref{thm2}, subject to Lemma \ref{pconst}, \ref{simplexkp}, and \ref{metricnet}. 
The proofs of Lemma \ref{pconst} and \ref{metricnet} are contained in the next two sections, and the proof of Lemma \ref{simplexkp} is contained in the last section. 

\section{The proof of Lemma \ref{pconst}}
\label{psepar}
For an arbitrary $k$-plane $P \in \calh$, use the natural parametrization 
$P=\{(t,g_P(t)): t \in \rr^k \}$ with $t=(t_1,\dots,t_k), g_P(t)=a^0(P)+t_1b^1(P)+\dots+t_kb^k(P).$ 
It is easy to see that for each $P \in \calh$, $f_P: \rr^k \to P, t \mapsto (t,g_P(t))$ is bi-Lipschitz, moreover, the Lipschitz constants are universally bounded. 
This means, there exist $L_1, L_2>0$ (depending only on $n,k$) such that 
\begin{equation}
\forall t_1,t_2 \in \rr^k, \forall P \in \calh, \ L_1 |t_1-t_2| \leq |f_P(t_1)-f_P(t_2)| \leq L_2 |t_1-t_2|. 
\label{eqbilip}
\end{equation}

Now we fix a $k$-plane $P \in \ti{E} \su \calh$, put $\ib=\hau_{\infty}^{\al} (P \cap B_l)$, 
and for an arbitrary $A \su \rr^n$, let $\ih(A)=\hau_{\infty}^{\al} (A \cap B_l)$. 
Recall that for any points $z_0,\dots,z_m \in [0,1]^n$,  $\De(z_0,\dots,z_m)$ denotes their convex hull, 
and that $\proj_0$ denotes the orthogonal projection onto the horizontal $k$-plane $V_0$.
Moreover, recall from \eqref{eqmep} that $m=\lceil \al \rceil \leq k$, and  $\ep = \al-(m-1) >0$. 

We will prove the following lemma, which is a core element of the proof. 
The idea for the construction, and thus a considerable part of the proof is due to Tam\'as Terpai. 

\begin{lemma}
\label{terp}
There exist compact sets $T_0,\dots,T_m \su P \cap [0,1]^n$ with the following properties: 
\begin{enumerate}[(i)]

	\item $\proj_0(T_i)$ is an axis-parallel cube of side length $ \approx \ib^{\phi}$ for all $i=0,\dots, m$,
	\label{te1}
	
	\item $\ih(T_i) \gkb \ib^{\psi}$ for all $i=0,\dots,m$, 
	\label{te2}
	
	\item for any $z_0 \in \proj_0(T_0),\dots,z_m \in \proj_0(T_m)$, $\hau^m(\De(z_0,\dots,z_m)) \gkb \ib^{\phi}$,
	\label{te3}
\end{enumerate}

where $\psi=1+\frac{mk}{\ep}, \phi=\frac{m}{\ep}$.
\end{lemma}

\begin{proof}
The construction of the sets $T_0,\dots, T_m$ is based on a suitable greedy method. 

We will use the following easy lemma. Recall that for any $V \su \rr^n$ and $d>0$, $V_d$ denotes the open $d$-neighborhood of $V$. 

\begin{lemma}
\label{hau}
For all $(m-1)$-dimensional affine subspace $V \su \rr^n$, $\ih(V_d) \leq \frac{\ib}{2}$, where $0<d=\ib^{\frac{1}{\ep}}C$ with some constant $C$ depending only on $n,\al$. 
\end{lemma}

\begin{proof}
Since for any $d>0$ and any $(m-1)$-dimensional subspace $V$, $V_d \cap B_l$ can be covered by $\approx (\frac{1}{d})^{m-1}$ many 
$(m-1)$-dimensional cubes of side length $d$, 
 we have 
$$\ih(V_d) \leq \frac{K}{d^{m-1}} \cdot d^{\al} n^{\al/2} =  d^{\ep} \cdot K'$$ 
for some constants $K,K'$ depending only on $n,\al$. 
We obtain $\ih(V_d) \leq \frac{\ib}{2}$ for $d = \ib^{\frac{1}{\ep}}C$, where 
$\ep=\al - (m-1)>0$, and $C$ is a constant depending only on $n,\al$ so we are done. 
\end{proof}

Fix $d$ from Lemma \ref{hau}. Note that clearly, $d \leq \ib \leq \hau^{\al}_{\infty} (B)$ which is at most $1$ by \eqref{bounded} of Lemma \ref{ass}. 
Put 
\begin{equation}
d_0=\frac{d}{2L_2} \leq d, 
\label{eqd0}
\end{equation}
where $L_2$ is from \eqref{eqbilip}. 
We fix another parameter 
\begin{equation}
r= \frac{d_0^m}{2^{k+1}\cdot k \cdot m!} \approx d^m.
\label{eqrdef}
\end{equation}

Divide $[0,1]^n \cap V_0$ into axis-parallel $k$-dimensional cubes of side length $r$, so $[0,1]^n \cap V_0=\bigcup_{j=1}^N K_j,$
 where $K_j$ is an axis-parallel closed $r$-cube in $V_0$, except from a few cubes near the boundary, which might be smaller, 
$N=(\lceil \frac{1}{r} \rceil)^k \leq \frac{2^k}{r^k}$ by $r \leq 1$. 
For each $K_j$ we denote its center with $c_j$. 
Put 
\begin{equation}
Q_j=f_P(K_j) \su P
\label{eqqj}
\end{equation}
for $j=1,\dots,N$. 

We will choose $T_0=Q_{j_0},\dots,T_m=Q_{j_m}$ from the set of $Q_j$'s in an appropriate greedy way with the following properties: 

\begin{enumerate}[(I)]
	\item $\ih(T_i) \geq \ib \frac{r^k}{2 \cdot 2^k}$ for all $i=0,\dots,m$, 
	\label{ll1}
	\item for each $i=2,\dots,m$, $d(p_i, H_{i-1}) \geq d_0$, where 
	$H_{i-1}$ denotes the affine subspace generated by the centers $p_0,\dots,p_{i-1}$ of $K_{j_0},\dots,K_{j_{i-1}}$. 
	\label{ll2}
\end{enumerate}

We will define $T_0$ to be $Q_{j_0}=f_P(K_{j_0})$ with the largest $\ih$-measure (which might be infinite). Let 
$$a_0=\max\{\ih(Q_j): j \in \{1,\dots,N\}\},$$ 
where we use the convention that $\max\{n_1,\dots,n_k,\infty\}=\infty$ for any real numbers $\{n_i\}_{i=1}^k$. 
Let $K_{j_0}$ be a cube such that  
$\ih(Q_{j_0})=a_0$, where $Q_{j_0}=f_P(K_{j_0})$ defined in \eqref{eqqj}, and put $T_0=Q_{j_0}$. Note that $\proj_0(T_0)=K_{j_0}$. The center of $K_{j_0}$ will be denoted with $p_0$. 
Clearly, $\ih(T_0)=a_0 \geq \ib \frac{r^k}{2^k}$, otherwise we would get 
$$\ih([0,1]^n \cap P)=\ih(f_P([0,1]^n \cap V_0)) \leq \sum_{j=1}^N \ih(f_P(K_j))< \frac{2^k}{r^k} \cdot \ib \frac{r^k}{2^k} =\ib,$$ 
which is a contradiction. Then \eqref{ll1} is clear for $i=0$. 

In the next step we will choose $T_1$ to be $Q_{j_1}$ which has the largest $\ih$-measure of those $Q_j$'s whose projections are at least $d_0$-away from $p_0$. 
Let 
$$a_1=\max\{\ih(Q_j): j \in \{1,\dots,N\}, d(c_j, p_0) \geq d_0 \},$$ 
choose a cube $K_{j_1}$ with 
$\ih(Q_{j_1})=a_1$, and put $T_1=Q_{j_1}$. 
The center of $K_{j_1}$ will be denoted with $p_1$. 
We claim that $\ih(T_1)=a_1 \geq \ib \frac{r^k}{2 \cdot 2^k}$. 

Let $D$ denote the $k$-dimensional ball $B(p_0,2d_0) \cap V_0$. 

By \eqref{eqbilip} and \eqref{eqd0}, $f_P(D)$ is contained in the ball $B(f_P(p_0),d) \cap P$, 
thus also contained in the $d$-neighborhood of some $(m-1)$-dimensional subspace $V$, so 
$\ih(f_P(D)) \leq \ih(V_{d}) \leq \frac{\ib}{2}$ by Lemma \ref{hau}. Clearly, 
\begin{align*}
\ih([0,1]^n \cap P) & \leq \ih(f_P(([0,1]^n \cap V_0) \setminus D)) + \ih(f_P(D)).
\end{align*}

We claim that 
\begin{equation}
([0,1]^n \cap V_0) \setminus D \su \bigcup_{j: d(c_j, p_0) \geq d_0} K_j.
\label{eqcov}
\end{equation}
Indeed, by \eqref{eqrdef}, we have $r < \frac{d_0}{\sqrt{k}}$, which implies that 
$$B(p_0,d_0+\sqrt{k}r) \cap V_0 \su D,$$
and it is easy to see that this implies the above claim. 

Then $a_1 < \ib \frac{r^k}{2 \cdot 2^k}$ would imply that 
\begin{align*}
\ih([0,1]^n \cap P) \leq \sum\limits_{j: d(c_j, p_0) \geq d_0} \ih(f_P(K_j)) + \ih(f_P(D)) < \frac{2^k}{r^k} \cdot \ib \frac{r^k}{2 \cdot 2^k} +  \frac{\ib}{2} =\ib,
\end{align*}
which is a contradiction. So we verified \eqref{ll1} for $i=1$. 

We proceed using induction on $i \in \{0,1,\dots,m \}$. 
Assume that $T_0,\dots, T_{i-1}$ are defined, $\ih(T_l) \geq \ib \frac{r^k}{2 \cdot 2^k}$ for each $l=0,\dots,{i-1}$, and 
$d(p_{l}, H_{l-1}) \geq \de$ for each $l=2,\dots,{i-1}$, where $p_{l}$ denotes the center of $K_{j_l}$, and $H_{l-1}$ 
denotes the affine subspace generated by the centers $p_0,\dots,p_{l-1}$ of $K_{j_0},\dots,K_{j_{l-1}}$.  
Let $H_{i-1}$ denote the affine subspace generated by  $p_0,\dots,p_{i-1}$.
Then we put 
$$a_{i}=\max\{\ih(Q_j): j \in \{1,\dots,N\}, d(c_j, H_{i-1}) \geq d_0 \},$$ choose a cube $K_{j_i}$ with 
$\ih(Q_{j_i})=a_i$, and put $T_{i}=Q_{j_i}$. We claim that $\ih(T_{i})=a_{i} \geq \ib \frac{r^k}{2 \cdot 2^k}$. 
Indeed, by $i \leq m$, by \eqref{eqbilip}, the $f_P$-image of the $2d_0$-neighborhood of $H_{i-1}$ is contained in the $d$-neighborhood of some $(m-1)$-dimensional subspace $V$, so 
$\ih((H_{i-1})_{2d_0}) \leq \ih(V_{d}) \leq \frac{\ib}{2}$ by Lemma \ref{hau}. 
Clearly, 
$$\ih([0,1]^n \cap P) \leq \ih(f_P(([0,1]^n \cap V_0)  \setminus (H_{i-1})_{2d_0})) + \ih(f_P((H_{i-1})_{2d_0}).$$
and similarly as in \eqref{eqcov}, we obtain that 
$$([0,1]^n \cap V_0) \setminus (H_{i-1})_{2d_0} \su \bigcup_{j: d(c_j, H_{i-1}) \geq d_0} K_j,$$
thus $a_{i+1} < \ib \frac{r^k}{2 \cdot 2^k}$ would imply that 

$$\ih([0,1]^n \cap P) \leq \sum\limits_{j: d(c_j,  H_{i-1}) \geq d_0} \ih(f_P(K_j)) + \ih(f_P((H_{i-1})_{2d_0})) 
< \frac{2^k}{r^k} \cdot \ib \frac{r^k}{2 \cdot 2^k} +  \frac{\ib}{2} =\ib,$$ 
which is a contradiction. So \eqref{ll1} holds for $i$. 

Property \eqref{ll2} is also clearly satisfied by the definition of $T_{i}$. 

Now we proceed with the proof of Lemma \ref{terp}. 
We claim that $T_0,\dots,T_m$ defined above are suitable for Lemma \ref{terp}. 
Clearly, since $T_r$ is the Lipschitz image of a closed cube, it is compact for each $r=0,\dots,m$.  
Recall that $d=\ib^{\frac{1}{\ep}}C$ with some constant $C$ depending only on $n,\al$. 

The side length of $\proj_0(T_i)$ is $r \approx d^m \approx \ib^{\phi}$ with $\phi=\frac{m}{\ep}$ for each $i=0,\dots,m$, so property \eqref{te1} of Lemma \ref{terp} is clear. 

By \eqref{ll1}, 
$$\ih(T_i)\geq \ib \frac{r^k}{2 \cdot 2^k} \approx \ib \cdot d^{mk} \approx \ib^{\psi}$$ with $\psi=1+\frac{mk}{\ep}$ for each $i=0,\dots,m$, so property \eqref{te2} of Lemma \ref{terp} is also clear.

Now we verify property \eqref{te3} of Lemma \ref{terp}. By construction, it is clear that 
$\hau^m(\De(p_0,\dots,p_m)) \geq \frac{d_0^m}{m!}$, where $p_i$ denotes the center of the cube $K_{j_i}$ for each $i$. 
Now we apply Lemma \ref{simplexkp} for $x=r, f(x)=\frac{d_0^m}{m!}$, and the axis-parallel $r$-cubes $K_{j_0},\dots,K_{j_m}$. By 
\eqref{eqrdef}, we have $x \leq \frac{1}{2^{k+1}k}f(x) \leq \frac{m}{2^{m+1}k}f(x)$, so the lemma indeed can be applied, and then for any 
$z_0 \in K_{j_0}, \dots, z_m \in K_{j_m}$, $\hau^m(\De(z_0,\dots,z_m)) \geq \frac{d_0^m}{2^{m+1}m!} \gkb d^m$. 

That is, for any $z_0 \in K_{j_0}=\proj_0(T_0), \dots, z_m \in K_{j_m}=\proj_0(T_m)$, 
$$\hau^m(\De(z_0,\dots,z_m)) \gkb d^m \gkb \ib^{\phi}$$ with $\phi=\frac{m}{\ep}$ and property \eqref{te3} of Lemma \ref{terp} is verified, so we are done with the proof of Lemma \ref{terp}. 
\end{proof}

Now we finish the proof of Lemma \ref{pconst}. Take $T_0=T_0(P),\dots,T_m=T_m(P)$ obtained from Lemma \ref{terp}. 
By \eqref{est}, \eqref{eqdela}, and \eqref{te2} of Lemma \ref{terp}, 
$$\ih(T_i(P))=\hau^{\al}_{\infty}(B_l \cap T_i(P)) \gkb \ib^{\psi} \gkb \la^{\psi}$$
for each $i=0,\dots,m$, where $\psi=1+\frac{mk}{\ep}$, so \eqref{w2} of Lemma \ref{pconst} is verified. 

By \eqref{te3} of Lemma \ref{terp}, we also have that for any $z_0 \in \proj_0(T_0),\dots,z_m \in \proj_0(T_m)$, 
$$\hau^m(\De(z_0,\dots,z_m)) \gkb \ib^{\phi} \geq \la^{\phi},$$ 
where $\phi=\frac{m}{\ep}$, so \eqref{w1} of Lemma \ref{pconst} is also verified. 

This means, we are done with the proof of Lemma \ref{pconst}. 

\section{The proof of Lemma \ref{metricnet}}
\label{linalg}

Recall that $e_0=(0,\dots,0)$; $e_1, \dots, e_n$ are the standard basis vectors of $\rr^n$, 
 $V_0$ is the $k$-dimensional linear space generated by $e_1, \dots, e_k$, $H_0=V_0^{\bot}$,   
$H_i=e_i + H_0$, $i=1,\dots,k$, and  
$$\calh=\{P \in A(n,k): P \cap H_i=\{x_i\} \ \text{for some} \ x_i \in [0,1]^n \ \forall \ i=0,1,\dots,k\}.
$$
For our convenience, we define a new metric on $\calh$. For $P \in \calh$, let 
$P=\{(t,g_P(t)): t \in \rr^k \}$ with $t=(t_1,\dots,t_k), g_P(t)=a^0(P)+t_1b^1(P)+\dots+t_kb^k(P)$ be the standard parametrization.

Note that by definition, $P \cap H_0=\{(0,a^0(P))\}$, where $0 \in \rr^k$, $a^0(P) \in \rr^{n-k}$, 
and if $P \cap H_i=\{(1^i,a^i(P))\}$, where $1^i=\proj_{\rr^k}(e_i) \in \rr^k, a^i(P) \in \rr^{n-k}$, then 
$b^i(P)=a^i(P)-a^0(P) \in \R^{n-k}$ for each $i=1,\dots,k$
Put 
$$x(P)=(a^0(P),b^1(P),\dots,b^k(P))=(a(P),b(P)) \in \rr^{(k+1)(n-k)}.$$
Clearly, $P \to x(P)$ is well defined and injective on $\calh$. 
We say that $x(P)$ is the ,,code" of $P$, and $\rr^{(k+1)(n-k)}$ is the ,,code space". 

We will use the maximum metric on $\rr^{(k+1)(n-k)}$. 
This means, let $$\|x(P)-x(P') \|=\max(\|a(P)-a(P')\|, \|b(P)-b(P')\|),$$ where
$$\|a(P)-a(P')\|=\max\limits_{j=1,\dots,n-k} |a_j^0(P)-a_j^{0}(P')|,$$ 
\begin{equation}
\label{cb}
\|b-b'\|=\max\limits_{j=1,\dots,n-k} \left(\max\limits_{i=1,\dots,k} |b_j^i(P)-b_j^{i}(P')| \right).
\end{equation}

Put 
\begin{equation}
d(P,P')=\|x(P)-x(P') \|. 
\label{eqdmet}
\end{equation}

We will prove the following lemma, which easily implies Lemma \ref{metricnet}. 

\begin{lemma}
\label{codelinalg}
 The metric space $(\mathcal{A},d)$ possesses a finite $\ep$-net $\calc \su \cala$ of cardinality $\approx \de^{-(k-m)(n-k)}$ 
with $\ep=\frac{K\de}{\hau^m(\De)}$, where $K$ is a constant depending only on $n,k,\al$.

That is, there exists a collection of $k$-planes 
$\calc \su \cala$ with $|\calc| \approx \de^{-(k-m)(n-k)}$ such that 
for any $P \in \cala$,
 $$x(P) \in \bigcup_{C \in \calc} B_{\| \cdot \|}\left(x(C), \frac{K\de}{\hau^m(\De)}\right),$$
where $K$ is a constant depending only on $n,k,\al$.
\end{lemma}

\begin{remark}
\label{metric}
It is easy to check that there exists a positive constant $K'$ such that, for every $P,P' \in \calh$,
$\rho (P,P') \leq K' \cdot d(P,P'),$ where $d$ is defined in \eqref{eqdmet}.
Lemma \ref{codelinalg} implies that for any $P \in \calh$ with $D_{j_r} \cap P \neq \emptyset$ for all $r=0,\dots,m$,
 $$P \in \bigcup_{C \in \calc} B_{\rho}\left(C, \frac{K' \cdot K\de}{\hau^m(\De)}\right),$$
thus Lemma \ref{codelinalg} indeed implies Lemma \ref{metricnet}. 
\end{remark}

\begin{proof}
Recall that  $c_r$ denotes the center of the dyadic cube $D_{j_r}$, $v_r=\proj_0 c_r$ $(r=0,\dots,m)$, and 
$\De=\De(v_0,\dots,v_m)$. 
We define a collection of $k$-planes 
$\calc \su A(n,k)$ with $|\calc| \approx \de^{-(k-m)(n-k)}$ such that each  $C \in \calc$ contains $\{c_r \}_{r=0}^m$,  and
if any $k$-plane $P$ 
intersects $D_{j_r}$ for all $r=0,\dots,m$, then $P$ must be contained in the $\frac{K \de}{\hau^m(\De)}$-neighborhood of $C$ for some $C \in \calc$.

Put 
\begin{equation}
U=\de \cdot \mathbb{Z}^{n-k} \cap [0,1]^{n-k},
\label{equi}
\end{equation}
that is, $U$ is a $\de$-net for the usual metric in $[0,1]^{n-k}$.

We will define the $\ep$-net $\calc$ using copies of the $\de$-net $U$ contained in 
$e_i + (\{0\} \times [0,1]^{n-k})$ for some suitably chosen $e_i$'s. 

We will need the following easy lemma. 
\begin{lemma}
\label{pontok}
Let $e_0=0$, and $e_1,\dots,e_k$ denote the standard unit vectors of $\rr^k$. Let $1 \leq m \leq k$ be integer, fix $v_0,\dots,v_m \in [0,1]^k$, and assume that 
$\hau^m(\De(v_0,\dots,v_m))=a>0$. Then there are $i_1,\dots, i_{k-m} \in \{0,1,\dots,k\}$ such that  
$$\hau^k(\De(v_0,\dots,v_m,e_{i_1},\dots,e_{i_{k-m}})) \gkb_{k,m} a.$$
\end{lemma}

We postpone the proof of Lemma \ref{pontok} to Section \ref{pure}. 

Fix $i_1,\dots, i_{k-m}$ obtained from Lemma \ref{pontok} for the projections 
$v_0,\dots,v_m$ of the centers of the cubes $D_{j_0},\dots,D_{j_m}$. We can assume without loss of generality that $i_1=1,\dots,i_{k-m}=k-m$. 
Put $\De'=\De(v_0,\dots,v_{m},e_1,\dots,e_{k-m})$. Then by Lemma \ref{pontok},

\begin{equation}
\hau^k(\De') \gkb_{k,m} \hau^m(\De) >0,
\label{eqkm}
\end{equation}
where $\De=\De(v_0,\dots,v_m)$ defined in \eqref{eqadelta}.

Define $C_{u^1,\dots,u^{k-m}}$ to be the $k$-dimensional affine subspace containing 
$c_0,\dots,c_m$, as well as $(e_1,u^1),\dots, (e_{k-m},u^{k-m}) \in [0,1]^n$ for some $u^{\om} \in U$, $\om=1,\dots,k-m$, 
where $U$ is defined in \eqref{equi}. 
Put 
\begin{equation}
\mathcal{C}=\{C_{u^1,\dots,u^{k-m}}: u^{\om} \in U, \om=1,\dots,k-m\}.
\label{eqcalc}
\end{equation}
Then we have $|\calc| \approx (\frac{1}{\de^{n-k}})^{k-m}=\de^{-(k-m)(n-k)}$, and by definition, $\calc \su \cala$. 

We claim that $\calc$ is an $\ep$-net for the metric space $(\mathcal{A},d)$ with $\ep=\frac{K\de}{\hau^m(\De)}$ for some constant $K$ 
depending only on $n,k,\al$. 

 For an arbitrary $k$-plane $P \in \calh$, put the parametrization 
\begin{equation}
P=\{(t,g_P(t)): t=(t_1,\dots,t_k) \in \rr^k \},
\label{eqplane}
\end{equation} 
where
\begin{equation}
\label{eqgp}
g_P(t)=a^0(P)+t_1b^1(P)+\dots+t_kb^k(P).
\end{equation}
Note that $(a^0(P),b^1(P),\dots,b^k(P))$ is precisely the code point of $P$ defined at the beginning of this section. 

We will need the following geometrical lemma. For an arbitrary vector $x \in \rr^{n-k}$, let $x_j$ denote its $j$'th coordinate, $j=1,\dots,n-k$.

\begin{lemma}
\label{rigid}
Let $s^0,\dots,s^k \in [0,1]^k$, $0 < \de < 1$, and $P,Q \in \calh$, where $\calh$ is defined in \eqref{hor}, such that  
$$\hau^k(\Om)>0, \ \text{where} \ \Om=\De(s^0,\dots,s^k),$$ 
\begin{equation}
|p^i_j-q^i_j| \lkb_{n,k} \de \ (j=1,\dots,n-k, i=0,\dots,k),
\label{eqpqdel}
\end{equation}
where $p^i=g_P(s^i), q^i=g_Q(s^i),$ and $g_P, g_Q$ denote the parametrization of $P, Q$ from \eqref{eqgp}, respectively. 
 
Then 
$$| a_j^0(P)-a_j^0(Q) |, |b_j^i(P)-b_j^i(Q) | \lkb_{n,k} \frac{\de}{\hau^k(\Om)} \ (j=1,\dots,n-k, \ i=1,\dots,k),$$
where $a^0(P),a^0(Q), b^1(P),b^1(Q),\dots,b^k(P),b^k(Q)$ are the coefficients from \eqref{eqgp}.
\end{lemma}

We postpone the proof of Lemma \ref{rigid} to Section \ref{pure}. 

Let $c_r=(v_r,w^r) \in [0,1]^k \times [0,1]^{n-k}$ denote the center of $D_{j_r}$ for $r=0,\dots,m$.
For $i=0,\dots,m$, put $s^i=v_i$, and for $i=m+1,\dots,k$, put 
$s^i=e_{i-m}$. 
Fix $P \in \cala$, and put 
\begin{equation}
z^r(P)=g_P(v_r) \in [0,1]^{n-k}, \ r=0,\dots,m, 
\label{eqzp}
\end{equation}
where $g_P$ is defined in \eqref{eqgp}. 
It is easy to see by the definition of $\cala$, see \eqref{eqcala}, that 
\begin{equation}
\ |z^r_{j} (P)- w^r_{j} | \lkb_{n,k} \de \ (j=1,\dots,n-k).
\label{eqdelta}
\end{equation}

Put
\begin{equation}
h^{\om}(P)=g_P(e_{\om}) \in [0,1]^{n-k}, \ {\om}=1,\dots,k-m, 
\label{eqxdef}
\end{equation}  
where $g_P$ is defined in \eqref{eqgp}. 
By the definition of $U$ \eqref{equi}, for any $P \in \cala$ we can choose $C=C_{u^1,\dots,u^{k-m}} \in \calc$ such that 
\begin{equation}
|h^{\om}_{j}(P)-u^{\om}_{j}| \leq \de \ (\om=1,\dots,k-m, \ j=1,\dots,n-k),
\label{eqxu}
\end{equation} 
where  $u^{\om} \in U$. 

We will apply Lemma \ref{rigid} for $s^0,\dots,s^k$, $P \in \cala$, and the above $C \in \calc$ in place of $Q$. 
Note that using the notations from Lemma \ref{rigid} and from \eqref{eqzp}, \eqref{eqxdef}, we have 
\begin{align*}
p^i=z^i(P), \ q^i=w^i \ &\text{for} \ i=0,\dots,m, \ \text{and} \\ 
p^i=h^{i-m}(P), \ q^i=u^{i-m} \ &\text{for} \ i=m+1,\dots,k. 
\end{align*} 
Clearly, using the notations from Lemma \ref{rigid}, \eqref{eqdelta} means that 
$$|p^i_j-q^i_j| \lkb_{n,k} \de \ (j=1,\dots,n-k) \ \text{for} \ i=0,\dots,m,$$ and 
\eqref{eqxu} means that 
$$|p^i_j-q^i_j| \leq \de \ (j=1,\dots,n-k) \ \text{for} \ i=m+1,\dots,k.$$ 
By \eqref{eqkm}, we also have $\hau^k(\Om)>0$, where 
$$\Om=\De(s^0,\dots,s^k)=\De(v_0,\dots,v_{m},e_1,\dots,e_{k-m})=\De'.$$

We checked that the conditions in Lemma \ref{rigid} are satisfied, so the lemma can be applied. 
We obtain that 
\begin{equation}
| a_j^0(P)-a_j^0(C) |, |b_j^i(P)-b_j^i(C) | \lkb_{n,k} \frac{\de}{\hau^k(\De')} \ (j=1,\dots,n-k, \ i=1,\dots,k),
\label{eqconc}
\end{equation}
where $a^0(P),b^1(P),\dots,b^k(P)$ are the coefficients from \eqref{eqgp}, and similarly, \linebreak
$a^0(C),b^1(C),\dots,b^k(C)$ are the coefficients from the parametrization  
$$C=\{(t,g_C(t)): t=(t_1,\dots,t_k) \in \rr^k \}, \ g_C(t)=a^0(C)+t_1b^1(C)+\dots+t_kb^k(C).$$
By \eqref{eqconc} and \eqref{eqkm}, we have 
$$| a_j^0(P)-a_j^0(C) |, |b_j^i(P)-b_j^i(C) | \lkb_{n,k,\al} \frac{\de}{\hau^m(\De)} \ (j=1,\dots,n-k, \ i=1,\dots,k).$$

By the definition of $\| \cdot \|$, this is equivalent to  
 $$x(P) \in B_{\| \cdot \|}\left(x(C), \frac{K\de}{\hau^m(\De)}\right),$$
where $K$ is a constant depending only on $n,k,\al$.
Thus we are done with the proof of Lemma \ref{codelinalg} (subject to Lemma \ref{pontok}, and \ref{rigid}), and so with the proof of Lemma \ref{metricnet} as well. 

\end{proof}

\section{The proofs of some purely geometrical lemmas}
\label{pure}

\begin{proof}[The proof of Lemma \ref{simplexkp}]
\hfill\\
First we prove that if there are $y_0 \in D_0,\dots,y_m \in D_m$ such that 
\begin{equation}
\hau^m(\De(y_0,\dots,y_m)) \geq f(x),
\label{eqy0}
\end{equation}
 then
\begin{equation}
 \forall \  z_0 \in D_0, \ \hau^m(\De(z_0,y_1,\dots,y_m)) \geq \frac{f(x)}{2}.
\label{eqz0}
\end{equation}

Let $H$ denote the $(m-1)$-dimensional affine subspace containing $y_1,\dots,y_m$. Clearly, 
$$\hau^m(\De(y_0,\dots,y_m))=\frac{\hau^{m-1}(\De(y_1,\dots,y_m)) \cdot d(y_0,H)}{m},$$
which implies by $\De(y_1,\dots,y_m) \su [0,1]^k$, $\hau^{m-1}(\De(y_1,\dots,y_m)) \leq \sqrt{k}$, and \eqref{eqy0} that 
\begin{equation}
d(y_0,H) \geq f(x) \frac{m}{\sqrt{k}}.
\label{eqtrivi}
\end{equation}
We also have 
$$\hau^m(\De(z_0,y_1,\dots,y_m))=\frac{\hau^{m-1}(\De(y_1,\dots,y_m)) \cdot d(z_0,H)}{m}.$$ 
Since $z_0, y_0 \in D_0$, 
$$d(z_0,H) \geq d(y_0,H)-\sqrt{k}\cdot x.$$
We claim that $d(y_0,H)-\sqrt{k}\cdot x \geq d(z_0,H)/2$. Indeed, 
by $x \leq \frac{m}{2^{m+1}k}  \cdot f(x) \leq \frac{m}{2k} \cdot f(x)$ and \eqref{eqtrivi}, we obtain 
$$\sqrt{k} \cdot x \leq f(x) \frac{m }{2 \sqrt{k}} \leq \frac{d(y_0,H)}{2},$$
so the claim is verified, and $d(z_0,H) \geq \frac{d(y_0,H)}{2}$. 
Then clearly, 
$$\hau^m(\De(z_0,y_1,\dots,y_m)) \geq \frac{\hau^{m-1}(\De(y_1,\dots,y_m)) \cdot d(y_0,H)}{2m} \geq \frac{f(x)}{2}$$ 
and we are done with the first step of the proof.  
Note that the role of the cube in the above proof is independent of the choice $0 \in \{0,\dots,m\}$. 

Let $z_1 \in D_1$. We repeat the first step using \eqref{eqz0} in place of \eqref{eqy0}. We use that
 $x \leq \frac{m}{2^{m+1}k}  \cdot f(x) \leq \frac{m}{4k} \cdot f(x)$, and we obtain by exactly the same argument as above that
$d(z_1,H) \geq \frac{d(y_1,H)}{2}$, where $H$ denotes the $(m-1)$-dimensional affine subspace containing $z_0,y_2,\dots,y_m$. 
Thus 
$$\hau^m(\De(z_0,z_1,y_2\dots,y_m)) \geq \frac{\hau^{m-1}(\De(z_0,y_2,\dots,y_m)) \cdot d(y_1,H)}{2m} \geq \frac{f(x)}{4}.$$ 
We continue the process, choosing each $z_i \in D_i$, $i \in \{0,1,\dots,m\}$ after another, using induction, 
$$\hau^m(\De(z_0,\dots,z_{i-1},y_i,\dots,y_m)) \geq \frac{f(x)}{2^i},$$ 
and $x \leq \frac{m}{2^{m+1}k}  \cdot f(x) \leq \frac{m}{2^{i+1}k} \cdot f(x)$ to obtain that 
$$\hau^m(\De(z_0,\dots,z_{i-1},z_i,y_{i+1},\dots,y_m)) \geq \frac{f(x)}{2^{i+1}}.$$
For $i=m$ we obtain that for any $z_0 \in D_0,\dots, z_m \in D_m$, 
$$\hau^m(\De(z_0,\dots,z_m)) \geq \frac{f(x)}{2^{m+1}},$$
and we are done. 
\end{proof}

\begin{proof}[The proof of Lemma \ref{pontok}]
\hfill\\
Assume first that $m=k-1$. 
Let $Q$ denote the $m$-dimensional hyperplane containing $v_0,\dots,v_m$, and let $W_0$ be the hyperplane parallel to $Q$ going through the origin. Put $W_i=W_0 + e_i$ for $i=0,1,\dots,k$. 
Let $w$ denote the unit vector generating the orthogonal complement of $W_0$. Clearly, we have 
$$d(W_0,W_i)=|\langle e_i,w \rangle|=|w_i|,$$ 
where $\langle \cdot,\cdot \rangle$ denotes the standard scalar product in $\rr^k$. 
Moreover, since $w$ is a unit vector, there exists $i \in \{1,\dots,k\}$ such that $d(W_0,W_i)=|w_i| \geq \frac{1}{\sqrt{k}}$. Since $Q$ is parallel to $W_i$ for all $i$, this also implies that there exists $j \in \{0,i\}$ such that $d(Q,e_j)=d(Q,W_j) \geq \frac{1}{2\sqrt{k}}$. 

Take $e_{i_1}=e_j$, then clearly, $\hau^k(\De(v_0,\dots,v_m,e_{i_1})) \gkb_k a$ and we are done. 

If $m < k-1$, let $P$ denote the $m$-dimensional affine subspace containing $v_0,\dots,v_m$, take an arbitrary hyperplane $Q$ containing $P$, and repeat the process described above for $Q$. This yields an $e_{i_1}$ with $d(P,e_{i_1}) \geq d(Q,e_{i_1}) \geq \frac{1}{2\sqrt{k}}$, so clearly, 
$\hau^{m+1}(\De(v_0,\dots,v_m,e_{i_1})) \gkb_k a$. Then we repeat the process starting with $v_0,\dots,v_m,$ $e_{i_1}$ to obtain a good $e_{i_2}$, and so on. Clearly, the process yields $e_{i_1},\dots,e_{i_{k-m}}$ with $\hau^k(\De(v_0,\dots,v_m,e_{i_1},\dots,e_{i_{k-m}})) \gkb_{k,m} a$, so we are done. 
\end{proof}

\begin{proof}[The proof of Lemma \ref{rigid}]
\hfill\\
First we reformulate the equations $p^i=g_P(s^i)$, $q^i=g_Q(s^i)$ as matrix equations. Recall that for 
$x \in \rr^{n-k}$,  $x_j$ denotes its $j$'th coordinate $(j=1,\dots,n-k)$. We have 

\begin{equation}
\label{mm}
M \cdot y_j(P)=p_j, \ M \cdot y_j(Q)=q_j \ (j=1,\dots,n-k),
\end{equation}
where
\begin{equation}
M=
\begin{bmatrix}
1 & s_1^0 & \dots & s_k^0 \\ 
1 & s_1^1 & \dots & s_k^1 \\ 
\vdots & \vdots & \ddots & \vdots \\
1 & s_1^k & \dots & s_k^k \\ 
\end{bmatrix}
, \ y_j(P)= 
\begin{pmatrix}
a_j^0(P) \\ 
b_j^1(P) \\
 \vdots \\
b_j^k(P)
\end{pmatrix}
, \ y_j(Q)=
\begin{pmatrix}
a_j^0(Q) \\ 
b_j^1(Q) \\
 \vdots \\
b_j^k(Q)
\end{pmatrix},
\label{matrixdef}
\end{equation}
and
$$p_j=
\begin{pmatrix}
p_j^0 \\ 
p_j^1 \\
 \vdots \\
p_j^k
\end{pmatrix}, \ 
q_j=
\begin{pmatrix}
q_j^0 \\ 
q_j^1 \\
 \vdots \\
q_j^k
\end{pmatrix}.
$$

It is easy to see that 
\begin{equation}
\det (M) \approx_{k} \hau^k(\De(s^0,\dots,s^k))=\hau^k(\Om) >0.
\label{eqdet}
\end{equation}

Then we obtain using $\det M >0$ and Cramer's rule that

\begin{equation}
\label{cramer}
a_j^0(P)= \frac{\det M (1| p_j)}{\det M}, b_j^i(P)= \frac{\det M (i+1| p_j)}{\det M} (i=1,\dots,k), 
\end{equation}
where $M (i| p_j)$ denotes the matrix formed by replacing the $i$'th column of $M$ by the column vector $p_j$. 
Similarly, we obtain the analog formulas for $a_j^0(Q)$ and $b_j^i(Q)$ $(i=1,\dots,k)$. 

It follows easily with cofactor expansion, using \eqref{cramer}, and that $s^i \in [0,1]^k \ (i=0,\dots,k)$, that 
\begin{align*}
\label{eqedelta}
| a_j^0(P)-a_j^0(Q) | & = \left| \frac{\det M (1| p_j) - \det M (1| q_j)}{\det M} \right| \lkb_k  \frac{ \sum_{i=0}^{k} | p_j^i-q_j^i |}{\det M}.
\end{align*}
Then we use \eqref{eqpqdel} and \eqref{eqdet} to obtain that 
\begin{align*}
| a_j^0(P)-a_j^0(Q) |  \lkb_k \frac{ \sum_{i=0}^{k} | p_j^i-q_j^i |}{\det M} \lkb_{n,k} \frac{ \de}{\det M} \lkb_{n,k} \frac{\de}{\hau^k(\Om)}
\end{align*}
and similarly 
$$| b_j^i(P)-b_j^i(Q) | \lkb_{n,k} \frac{\de}{\hau^k(\Om)}$$
for each $j=1,\dots,n-k, i=1,\dots,k$, and we are done.  
\end{proof}

\section*{Acknowledgment}
The author is grateful to Tam\'as Keleti for the inspiring conversations, and for his careful reading and useful suggestions. 
The author would also like to thank Izabella {\L}aba for the helpful discussions which motivated the investigation 
of $(\al,k,s)$-Furstenberg sets for $0 <\al \leq k-1$ in addition to $k-1 < \al \leq k$, and 
Tam\'as Terpai for his contribution to Lemma \ref{terp}. 
The author thanks the anonymous referees for their valuable comments.


\begin{thebibliography}{30}

\bibitem{Bo03}
J. Bourgain, 
\newblock On the Erd\"os-Volkmann and Katz-Tao ring conjectures, 
\textit{Geom. Funct. Anal.}  \textbf{13} (2003), no. 2, 334--365.

\bibitem{Fa}
K. Falconer,
\newblock \textit{The Geometry of Fractal Sets}, 
\newblock Cambridge University Press, 1985


\bibitem{Fr}
D. H. Fremlin, 
\newblock \textit{Measure Theory: Topological Measure Spaces (Vol. 4)}, 
Torres Fremlin, 2003.

\bibitem{Fu}
H. Furstenberg, 
\newblock  Intersections of Cantor sets and transversality of semigroups, 
\newblock  \textit{Problems in analysis} (Sympos. Salomon Bochner, Princeton Univ., Princeton, N.J. 1969), Princeton Univ. Press, Princeton  N.J. (1970), 41--59.

\bibitem{HKM}
K. H\'era, T. Keleti and A. M\'ath\'e, 
Hausdorff dimension of unions of affine subspaces and of Furstenberg-type sets, accepted to J. Fractal Geom., arXiv:1701.02299.

\bibitem{Ho}
J. D. Howroyd,
\newblock On dimension and on the existence of sets of finite positive Hausdorff measure,
\newblock \textit{Proc. Lond. Math. Soc. (3)} \textbf{70} (1995), 581--604.

\bibitem{KaTa}
N. H. Katz, T. Tao, 
\newblock Some connections between Falconer’s distance set conjecture and sets of Furstenburg type, 
\newblock \textit{New York J. Math.} \textbf{7} (2001), 149--187.
 

\bibitem{LuSt}

N. Lutz, D.M. Stull, 
\newblock Bounding the Dimension of Points on a Line,  
\newblock In: \textit{Gopal T., Jäger G., Steila S. (eds)  TAMC 2017}, Lecture Notes in Comput. Sci., \textbf{10185}, 425–439, Springer, Cham (2017). 

\bibitem{Ma95}
P.~Mattila, 
\newblock \textit{Geometry of sets and measures in Euclidean spaces}, 
\newblock Cambridge University Press, 1995.

\bibitem{MR}
U.~Molter and E.~Rela,
\newblock Furstenberg sets for a fractal set of directions,
\newblock \textit{Proc. Amer. Math. Soc.} \textbf{140} (2012), 2753--2765.


\bibitem{Ob2}
D. M.~Oberlin, 
\newblock Exceptional sets of projections, unions of $k$-planes, and associated transforms, 
\newblock Israel J. Math. \textbf{202} (2014), 331--342.


\bibitem{Wo99}
T.~Wolff,
\newblock Recent work connected with the Kakeya problem,  
\newblock In: \textit{Prospects in Mathematics}, (1999), 129--162.

\end{thebibliography}
\end{document}